\documentclass[a4paper, reqno,12pt]{article}
\usepackage[margin=1.0 in  ]{geometry}         
\usepackage{graphicx}
\usepackage{stmaryrd}
\usepackage{framed}
\usepackage{amssymb}
\usepackage{epstopdf}
\usepackage{amsmath,amsfonts,amssymb,amsthm}
\usepackage{enumerate}
\usepackage{multirow}
\usepackage{fixltx2e}
\usepackage{mathptmx}

\usepackage{xcolor}
\usepackage{soul}

\usepackage{xcolor}
\usepackage{soul}

\newtheorem{thm}{Theorem}[section]
\newtheorem{cor}[thm]{Corollary}
\newtheorem{lem}[thm]{Lemma}

\newtheorem{ex}[thm]{Example}
\newtheorem{defn}[thm]{Definition}
\newtheorem{rem}[thm]{Remark}
\numberwithin{equation}{section}

\providecommand{\ds}{\displaystyle}
\providecommand{\dss}{\displaystyle \sum}
\providecommand{\dsf}{\displaystyle \frac}

\providecommand{\dsm}{\displaystyle \min}

\providecommand{\norm}[1]{\lvert#1\rvert}
\providecommand{\stb}[1]{\left\{#1\right\}}
\providecommand{\opint}[1]{\left]#1\right[}
\providecommand{\clint}[1]{\left[#1\right]}
\providecommand{\innp}[1]{\langle#1\rangle}

\providecommand{\ra}{\rightarrow}

\providecommand{\RA}{\Rightarrow}

\providecommand{\grad}{\nabla}

\providecommand{\eps}{\epsilon}
\providecommand{\om}{\omega}

\providecommand{\lam}{\lambda}
\providecommand{\bx}{\bar{x}}

\providecommand{\bk}[1]{\left(#1\right)}

  \providecommand{\conv}{\operatorname{conv}}     
\providecommand{\pt}{\partial}

\providecommand{\probi}{\text{Test~Problem 1}}
\providecommand{\probii}{\text{Test~Problem 2}}
\providecommand{\probiii}{\text{Test~Problem 3}}

\providecommand{\tuple}{\bk{\spy{0},\spy{1},\ldots,\spy{m}}}
\providecommand{\matt}{\RR^{m\times(m+1)}}
\providecommand{\fopt}{f_{\rm opt}}
\providecommand{\xopt}{x_{\rm opt}}
\newcommand{\ay}[3][y]{\ensuremath{{#1}_{#3,#2}}}
\newcommand{\spy}[2][y]{\ensuremath{{#1}_{#2}}}
\newcommand{\sby}[2][y]{\ensuremath{{#1}_{#2}}}
\newcommand{\Lf}[2][L]{\ensuremath{{#1}_{#2}}}
\newcommand{\Kf}[2][K]{\ensuremath{{#1}_{#2}}}

\newcommand{\DFOCM}{$\text{DFO}_\eps\text{CM}$}

\providecommand{\RR}{\mathbb{R}}
\providecommand{\NN}{\mathbb{N}}
\providecommand{\C}{\mathcal{C}}
\newcommand{\Cr}[2][S]{\ensuremath{{#1}_{#2}}}
\DeclareMathOperator*{\amin}{argmin}
\DeclareMathOperator*{\amax}{argmax}

\newcommand{\sepp}{\setlength{\itemsep}{-3pt}}

\title{A Derivative-Free CoMirror Algorithm}
\author{Heinz H. Bauschke\thanks{Mathematics. Irving K. Barber school, University of British Columbia,  Kelowna, B.C. V1V 1V7, Canada. {\tt Heinz.Bauschke@ubc.ca}. }, 
Warren L. Hare\thanks{Mathematics. Irving K. Barber school, University of British Columbia,  Kelowna, B.C. V1V 1V7, Canada. {\tt Warren.Hare@ubc.ca}. }, 
Walaa M.  Moursi\thanks{Mathematics. Irving K. Barber school, University of British Columbia,  Kelowna, B.C. V1V 1V7, Canada. {\tt Walaa.Moursi@ubc.ca}. 
}}

\date{October 23, 2012} 

\begin{document}

\maketitle

\begin{abstract} 
\noindent
We consider $\min\{f(x):g(x) \le 0, ~x\in X\},$
where $X$ is a compact convex subset of $\RR^m$, and $f$ and $g$
are continuous convex functions defined on an open neighbourhood
of $X$.  We work in the setting of derivative-free optimization,
assuming that $f$ and $g$ are available through a black-box that
provides only function values for a lower-$\mathcal{C}^2$ representation
of the functions.  We present a derivative-free
optimization variant of the $\eps$-comirror algorithm \cite{BBTGBT2010}.
Algorithmic convergence hinges on the ability to accurately approximate
subgradients of lower-$\mathcal{C}^2$ functions, which we prove
is possible through linear interpolation. We provide convergence
analysis that  quantifies the difference between the function values
of the iterates and the optimal function value.  We find that the
DFO algorithm we develop has the same convergence result as the
original gradient-based algorithm.  We present some 
numerical testing that demonstrate the practical feasibility of the
algorithm, and conclude with some directions for further research.
\end{abstract}


\section{Introduction}

In this paper we introduce a derivative-free linear interpolation-based method for solving constrained optimization problems of the form
\begin{equation}
(P):\min\{f(x):g(x) \le 0, x\in X\},
\end{equation}
where $f$ and $g$ are continuous convex functions defined on a
nonempty open convex subset $O$ of $\RR^m$, and where the constraint
set $X $ is a nonempty compact convex subset of $O$.  We further
assume that we have access to the lower-$\C^{2}$ 
representations of $f$ and $g$ and that
the problem is feasible i.e., there exists some $x_0\in X$ such
that $g(x_0)\le 0$. 
The algorithm is based on the $\eps$-comirror
algorithm presented in \cite{BBTGBT2010}.  Derivative-free optimization
(DFO) is a rapidly growing field of research that explores the
minimization of a black-box function when first-order information
(derivatives, gradients, or subgradients) is unavailable.  While
the majority of past work in DFO has focused on unconstrained
optimization, several methods have recently been introduced for constrained
optimization.  In constrained optimization, most of the analysis
of DFO methods has been done within the framework of {\em direct
search} and {\em pattern search} methods.  That is, methods that
do not attempt to build interpolation (or other such) models of the
objective function, but instead use concepts like positive bases
to ensure convergence.  Such methods can be adapted to constrained
optimization through techniques by e.g. projecting search directions
onto constraint sets \cite{LST02, hare-2010}, ``pulling back''
search directions onto manifolds \cite{Dreisigmeyer06a, Dreisigmeyer06b},
the use of filtering techniques \cite{AAD2007}, or barrier based
penalties \cite{Audet-Dennis-2006}.

On the other hand, fairly little research has explored approaching constrained
optimization via model-based DFO methods.  
Notable in this area is
\cite{Berghen04,Vanden-Berghen-Bersini-2005}, which extends the UOBYQA 
\cite{Powell02} to constrained optimization (in an algorithm
named CONDOR).  
This paper provides a novel model-based DFO method
for linearly constrained optimization.  
Our algorithm is
designed for constraints defined by a given convex function.

Our algorithm is based on the $\eps$-comirror algorithm 
\cite{BBTGBT2010}.  
The $\eps$-comirror algorithm finds its roots in mirror-descent methods \cite{Nemirovsky-Yudin-1983, Ben-Tal-M-N-2001, BeckTeboulle-2003}. These methods can be viewed as nonlinear projected subgradient methods that use a general distance-like function (the Bregman distance) instead of the usual Euclidean squared distance \cite{BeckTeboulle-2003}.  The $\eps$-comirror algorithm adapts the mirror-descent method to work for convex constrained optimization where the constraint set is provided by a convex function.  It requires that the problem is additionally constrained by a convex compact set and that the subgradients (of both the constraint function and the objective function) are bounded over this set.  

The algorithm presented here differs from previous research in two
other notable ways.  First, unlike past model-based DFO method, we
do not assume that the objective function is $\C^2$; instead, we
work with the broader class of lower-$\C^2$ functions  (see
definition~\ref{lowerc1}).  Lower-$\C^2$ functions include convex
\cite[Theorem~10.33]{rockafellar-wets-1998} and $\C^2$ functions
(by definition), as well as fully amenable functions
\cite[Exercise~10.36]{rockafellar-wets-1998} and finite max functions
(Example~\ref{ex:finitemax} below).  To work with lower-$\C^2$
functions, we develop a method to approximate subgradients for such
functions and analyze it for the derivative-free algorithm.  In
particular, in Theorem~\ref{errbdd} we define the approximate
subgradient for an arbitrary lower-$\C^2$ function and prove that
it satisfies an error bound analogous to the one introduced in
\cite[Theorem 2.11]{conn-scheinberg-vicente-2009} for the class of
$\C^1$ functions.

The second major difference from previous DFO research is that we
present a convergence result that quantifies the difference between the
function values of the iterates and the optimal function value.  
To the best of our knowledge, 
this provides the first results of this kind for a multivariable DFO method. 
It is remarkable that the DFO algorithm we develop has the same
convergence result as the original gradient-based algorithm presented
in \cite{BBTGBT2010}.
(A quadratically convergent DFO
method is developed in \cite{GhoshHager-1990}, but only for
functions defined on $\RR$.
Furthermore, in \cite{Mifflin}, a superlinearly convergent algorithm is
presented.) 

The remainder of this paper is organized as follows. Section~\ref{s:2}
is a brief introduction to the main building blocks we use. First,
we provide the definition of the class of lower-$\C^2$ functions
and some properties. Second, we provide the definition of the linear
interpolation model of a function $f$ over a subset $Y$ of $\RR^m$
and a sufficient condition to be well-defined. Finally, we give the
definition and the main properties of Bregman distances. In
Section~\ref{s:3} we give the first key result in Theorem~\ref{errbdd},
on which we build our convergence results. In Section~\ref{s:4} we
describe our derivative-free $\eps-$comirror algorithm. In
Theorem~\ref{t:conv} we establish the convergence analysis. In
Section~\ref{s:5} we provide some numerical results
that confirm the practical feasibility of the algorithm.
Section~\ref{s:conc} summarizes some concluding remarks.  To make
the presentation self-contained we add Appendix~\ref{s:app} which 
includes the proofs of two basic inequalities. 

\section{Auxiliary Results}\label{s:2}
We shall work in $\RR^m$, equipped with the usual 
Euclidean norm $\norm{\cdot}$. 
Throughout the remainder of the paper, we suppose that 
\begin{center}
$O$ is a nonempty open convex subset of $\RR^m$.
\end{center}
Recall that for a convex function $f:O\to \RR$, the subdifferential $\pt f$ at a point $x\in O$ is defined by
 \begin{equation}\label{subdiffer}
 \pt f (x)=\stb{v\in \RR^m: f(y)\ge f(x)+\innp{v,y-x} \text{    for all   } y\in O}.
 \end{equation}
We denote the {\em closed} ball in $\RR^m$ centred at $x_0$ with radius
$\Delta>0 $ by 
 $$B(x_0;\Delta)=\stb{x\in \RR^m:\norm{x-x_0}\le \Delta},$$ 
and the set of \emph{natural numbers} by 
\begin{equation*}
\NN=\stb{1,2,3,\dots}.
\end{equation*}
Given $r\in\NN$, we abbreviate the \emph{unit simplex} in $\RR^{r}$ by 
\begin{equation*}
\Cr{r}:=\Big\{\lam\in \RR^{r}: \sum_{i=1}^{r}\lam_i=1,\lam_i\in \clint{0,1},
i\in\{1,\ldots,r\}\Big\}.
\end{equation*}
Finally, we shall use $\norm{L}$ to denote the 
spectral norm of a matrix $L\in\RR^{m\times m}$.

\subsection*{The Class of Lower-$\C^k$ Functions}
We next introduce the class of lower-$\C^2$ functions.  

 \begin{defn}\label{lowerc1}~{\rm \cite[Definition 10.29]{rockafellar-wets-1998}}
 A function $f:O\to \RR$ is said to be a \emph {lower-$\C^k$ function} at $\bx\in O$ if there exists a neighbourhood $V=V(\bx)\subseteq O$ and a representation  
\begin{equation}\label{e:lc2}
f(x)=\ds\max_{t\in T} f_t(x) 
\end{equation}
in which all functions $f_t$ are of class $\C^{k}$ on $V$, the index set $T:=T(\bx)$ is a compact topological space, and $f_t$ and the first $k$ derivatives of $f_t$ depend continuously not just on $x\in V$ but even on $(t,x)\in T\times V$. In this case we say that \eqref{e:lc2} provides a {\em lower-$\C^k$ representation} of $f$ at $x\in O$. The function $f$ is said to be \emph{lower-$\C^{k}$ on $O$} if $f$ is lower-$\C^{k}$ at every point $x\in O$. \end{defn}

The next Lemma provides details regarding when a convex function is lower-$\C^2$. 

\begin{lem}\label{con_lc2}{\rm \cite[Theorem 10.33]{rockafellar-wets-1998}}
 Let $f:O\to \RR$ be convex. Then $f$ is lower-$\C^{2}$ on $O$.
\end{lem}

Although the class of lower-$\C^2$ functions includes many convex
functions \cite[Theorem~10.33]{rockafellar-wets-1998}, it should
be noted that our algorithm will require access to a lower-$\C^2$
representation of the objective and constraint functions.  The next
example shows that any finite max function is not only lower-$\C^2$,
but also provides a natural lower-$\C^2$ representation.

\begin{ex}\label{ex:finitemax}
 Let $f:O\to \RR$ be defined as $f=\max \stb{f_1,\dots,f_n}$, where each $f_i$ is of class $\C^k$ on $O$. Then $f$ is lower-$\C^k$ on $O$. (This is the case where $T$ is $\stb{1,\ldots,n} $ equipped with the discrete topology.)
\end{ex}

The value of working with lower-$\C^2$ functions is seen in Lemma~\ref{subgd}, which demonstrates how to compute the subdifferential of a lower-$\C^2$ function.  
	\begin{lem}\label{subgd}
 Let $f:O\to \RR$ be a convex function that has a lower-$\C^2$ representation $f(x)=\ds\max_{t\in T} f_t(x) $ at $\bar{x}\in O$ and set $A(\bar{x})=\ds\amax_{t\in T} f_t(\bar{x}) $. Then 
 \begin{equation*}
\pt f(\bar{x})=\conv\stb{\grad f_t (\bar{x})|~t\in A(\bar{x})}.
\end{equation*}
\end{lem}

\begin{proof}
Combine \rm\cite[Theorem 10.31]{rockafellar-wets-1998} and \rm\cite[Proposition 8.12]{rockafellar-wets-1998}.
\end{proof}

\begin{thm} {\rm \cite[Proposition~10.54]{rockafellar-wets-1998}}\label{save}
Let $f:O\to \RR$ be a lower-$\C^2$ function, and let $X$ be a nonempty
compact subset of $O$. Then there exists an open set $O'$ with $X\subseteq
O'\subseteq O$, such that $f$ has a common lower-$\C^2$ representation
valid at all points $x\in O'$, i.e., there exists a compact topological
space $T$, and a family of functions $(f_t)_{t\in T}$ defined on $O'$ such that 
\begin{equation}
f=\max_{t\in T}f_t \text{~~~~on ~~}O',
\end{equation}
 and the functions $(t,x)\mapsto f(t,x)$, $(t,x)\mapsto \grad f(t,x)$, and $(t,x)\mapsto \grad^2 f(t,x)$ are continuous on $T\times O'$.
\end{thm}

To prove convergence of the algorithm introduced in this paper, we require bounds on the subgradients of the objective and the constraint functions. Lemma~\ref{bound:aux} provides a proof of the existence of this bound.
\begin{lem}\label{bound:aux}
Let $f:O\to \RR$ be convex, and let $X$ be a nonempty compact subset of $O$. Then 
\begin{equation*}
\sup\norm{\pt f(X)}<+\infty.
\end{equation*}
\end{lem}

\begin{proof} Since $f$ is convex, Lemma \ref{con_lc2} implies that
$f$ is  lower-$\C^2$ on $O$. Since $X$ is a nonempty compact subset
of $O$, Theorem~\ref{save} guarantees the existence of an open
subset $O'$ with $X\subseteq O' \subseteq O$ such that $f$ has a
common lower-$\C^2$ representation valid at all points $x\in O'$.
Let $f=\max_{t\in T}f_t$ be as stated in Theorem~\ref{save}. The
definition of lower-$\C^2$ implies that the mapping $(t,x)\mapsto
\norm{\grad f_t(x)}$ is continuous on $T\times O'$. 
By the Weierstrass Theorem, 
$L := \max_{(t,x)\in T\times X} \norm {\grad f_t(x)} < \infty$.
Now, let $x\in X$, and let $v\in \pt f(x)$.
Using Lemma~\ref{subgd} we know that $v=\sum_{t\in A(x)} \lam _t
\grad f_t(x)$ for some $\lam\in S^{r}$ where $r\in\NN$ is the number of
elements in $A(x)$. Therefore
\begin{equation*}
\norm{v}=\bigg| \sum_{t\in A(x)} \lam _t \grad f_t(x) \bigg|\\
\le \sum_{t\in A(x)} \lam _t  \norm {\grad f_t(x) }\\
 \le \sum_{t\in A(x)} \lam _t L\\
 =L,
 \end{equation*} 
 and the proof is complete.
(Alternatively, one may consider either the lower semicontinuous hull of
$f$ and apply {\rm \cite[Theorem~24.7]{Rocky70}}, or use {\rm \cite[Corollary~12.38]{rockafellar-wets-1998}} after
extending $\partial f$ to a maximally monotone operator.)
\end{proof}

\begin{lem} \label{coffee}
Let $f:O\to \RR$ be a lower-$\C^2$ function, and let $X$ be a nonempty
compact convex subset of $O$. 
Let $O'$, $T$, and $(f_t)_{t\in T}$ be as in 
Theorem~\ref{save}. 
Then there exists $K_f \geq 0$ such that 
$\nabla f_t$ is $K_f$-Lipschitz on $O'$
for every $t\in T$. 
\end{lem}
\begin{proof}
By Theorem~\ref{save}, 
$(t,x)\mapsto\grad^2 f_t(x)$ is continuous on the compact set $T\times X$.
Therefore, by the Weierstrass theorem,
$K_f := \max_{(t,x)\in T\times X}\|\grad^2 f_t(x)\| <+\infty$. 
Now apply the Mean Value Theorem {\rm \cite[Theorem~5.1.12]{Denk03}}.

\end{proof}

\subsection*{The Linear Interpolation Model }
In our method we use a derivative-free model-based technique. Therefore, in this section we introduce the definition of the linear interpolation model and related facts.

\begin{defn}
Let $f:O\to \RR$ be a function, and let $Y= \tuple\in \matt$. 
If the matrix 
\begin{equation*}
Q=\begin{pmatrix}
1&\ay{1}{0}&\dots&\ay{m}{0}\\
1&\ay{1}{1}&\dots&\ay{m}{1}\\
  \vdots  & \vdots  & \ddots & \vdots  \\
1&\ay{1}{m}&\dots&\ay{m}{m}
\end{pmatrix}
\end{equation*}

is invertible, then $Y$ is said to be \emph{a poised tuple centred at $y_0$}. Moreover, if $\stb{y_0,y_1,\ldots,y_m}\subseteq O$ then $Y$ is said to be \emph{a poised tuple centred at $y_0$ with respect to $f$}. In this case the linear system 
\begin{equation*}
\begin{pmatrix}
1&\ay{1}{0}&\dots&\ay{m}{0}\\
1&\ay{1}{1}&\dots&\ay{m}{1}\\
  \vdots  & \vdots  & \ddots & \vdots  \\
1&\ay{1}{m}&\dots&\ay{m}{m}
\end{pmatrix}
\begin{pmatrix}
\alpha_0\\
\alpha_1\\
\vdots\\
\alpha_m
\end{pmatrix}
=
\begin{pmatrix}
f(\spy{0})\\
f(\spy{1})\\
\vdots\\
		f(\spy{m})
\end{pmatrix}.
\end{equation*}
 has a unique solution $(\alpha_0,\alpha_1,\ldots,\alpha_m)\in \matt$, and the \emph{Linear Interpolation Model} of the function $f$ over $Y$ is the unique (well defined) function 
 
\begin{equation*}
F\colon \RR^m \to \RR \colon x\mapsto \alpha_0+\ds\sum_{i=1}^{n}\alpha_i x_i.
\end{equation*}
Note that in this case $F$ satisfies the interpolation conditions
\begin{equation*}
F(\spy{i})=f(\spy{i}),~~ \text{for every}~~i\in \stb{0,1,\ldots,m}.
\end{equation*}
\end{defn}

The following Theorem provides the error bound satisfied by the approximate gradient of the linear interpolation model.

\begin{thm}{\rm\cite[Theorem 2.11]{conn-scheinberg-vicente-2009}}\label{Lin_int_err}
Suppose that $f:O\to \RR$ is $\C^2$ function on $O$. Let $y_0\in O$. Assume that $Y=\tuple \in \matt$ is a poised tuple of sample points centred at $y_0$ with respect to $f$. 
Set $ \Delta=\ds\max _{1\le i\le m}\norm{\spy{i}-\spy{0}} $. Suppose that $B (\spy{0};\Delta)\subseteq O$.
Let  $\grad f$ be $\Kf{f}$ Lipschitz over $B (\spy{0};\Delta)$. Then the gradient of the linear interpolation model $F$ satisfies an error bound of the form
\begin{equation*}
\norm{\grad f(y)-\grad F(y)}\le K \Delta,~~~\text{for all} ~~~y\in B (\spy{0};\Delta) ,
\end{equation*}
where 
\begin{equation}\label{mat_L_def}
 K:=K_f(1+\sqrt{m}\norm{\hat{L}^{-1}}/2), ~~~  L=L(Y):=
 \begin{pmatrix}
\spy{1}-\spy{0}\\
 \spy{2}-\spy{0}\\
 \vdots\\
\spy{m}-\spy{0}
 \end{pmatrix},
\text{   and  }\hat{L}=\hat{L}(Y):=\ds\frac{1}{\Delta}L. 
\end{equation} 
\end{thm}

\subsection*{The Bregman Distance: Definition and Properties}
The last building block used in our analysis is the Bregman distance.

\begin{defn}{\rm \cite{bregman67}}
 Let $\om:O\to \RR$ be a convex differentiable function. 
The corresponding \emph{Bregman distance} $D_\om$ is
\begin{equation}
D_{\om}\colon O\times O \to\RR\colon (u,v)\mapsto\om(u)-\om(v)-\innp{\grad \om (v),u-v}.
\end {equation}
\end{defn}
 \begin{defn}{\rm\cite[Section~3.5]{Zal02}}
 Let $C$ be a nonempty convex subset of $\RR^m$. Let $\om:C\to \RR$. Then $\om$ is said to be \emph{strongly convex } with convexity parameter $\alpha >0$, if for all $x,y\in C$, $t\in \clint{0,1}$ we have
\begin{equation*}
\om(tx+(1-t)y)\le t\om(x)+(1-t)\om(y)-\frac{\alpha}{2}t(1-t)\norm{x-y}^2.
\end{equation*}
\end{defn}
Throughout the next arguments we shall assume that $\om $ is a strongly convex and differentiable function on a nonempty convex subset of $ \RR^m$ , with a convexity parameter $\alpha>0$.
In this paper we shall be interested in Bregman distances that are created from strongly convex functions.

The following result is part of the folklore (and established in much
greater generality in e.g., \cite[Section~3.5]{Zal02}); 
for completeness we include the proof. 
\begin{lem}\label{equi}
Let $\om:O\to \RR$ be a differentiable function. Let $X$ be a nonempty subset of $O$. Then the following are equivalent:
\begin{enumerate} [\rm (i)]
	\item \label{i} 
	$\om(\lambda x+(1-\lambda)y) \le \lambda \om(x)+(1-\lambda)\om(y)-\frac{\alpha}{2}\lambda(1-\lambda){\norm{x-y}}^2$ for all $x,y \in X$ and $\lambda \in \opint{0,1}$.
	\item \label{ii}  
$D_\om(x,y)=\om(x)-\om(y)-\innp{\grad \om(y),x-y} \ge
\frac{\alpha}{2}{\norm{x-y}}^2$ for all $x,y \in X$ and $\lambda \in
\opint{0,1}$. 
	\item \label{iii} $\innp{\grad \om(x)-\grad \om(y),x-y}\ge \alpha {\norm{x-y}}^2$ for all $x,y \in X$ and $\lambda \in \opint{0,1}$. 
\end{enumerate}
\end{lem}

\begin{proof}
``\eqref{i}$\Rightarrow$\eqref{ii}'':
Rewrite (\ref{i}) as
\begin{equation}
 \om(y+\lambda(x-y)) \le \lambda \om(x)+(1-\lambda)\om(y)-\frac{\alpha}{2}\lambda(1-\lambda){\norm{x-y}}^2 .
\end{equation}
Hence 
\begin{equation*}
\frac{\om(y+\lambda(x-y)) -\omega (y)}{\lambda} \le \om(x)-\om(y)-\frac{\alpha}{2}(1-\lambda){\norm{x-y}}^2.
\end{equation*}
Taking the limit as $\lambda \ra 0^{+}$ and using the assumption that $\om$ is differentiable we see that
\begin{equation*}
\innp{\grad \om(y),x-y}\le \om(x)-\om(y)- \frac{\alpha}{2}{\norm{x-y}}^2.
 \end{equation*}
Hence \eqref{ii} holds.

 ``\eqref{ii}$\RA$\eqref{i}".
Suppose that (\ref{ii}) holds for all $ x,y \in X$. Let $\lam\in \opint{0,1}$. Set $z=\lambda x +(1-\lambda)y \in X$. Applying  (\ref{ii}) to $x$ and $z$ yields
\begin{equation}\label{ineq:1}
\om(z) \le \om(x)-\innp{\grad \om(z),x-z} - \frac{\alpha}{2}{\norm{x-z}}^2.
\end{equation}
Similarly, applying  (\ref{ii}) to $y$ and $z$ yields 
\begin{equation}\label{ineq:2}
\om(z) \le \om(y)-\innp{\grad \om(z),y-z} - \frac{\alpha}{2}{\norm{y-z}}^2.
\end{equation}
Multiplying \eqref{ineq:1} by $\lambda$ and \eqref{ineq:2} by $(1-\lambda)$, and adding we get
\begin{align*}
\om(z) & \le \lambda \om(x) +(1-\lambda )\om(y)-\lambda \innp{\grad \om(z),x-z} -(1-\lambda) \innp{\grad \om(z),y-z}\\
&  -\frac{\alpha}{2}\bk{\lambda {\norm{x-z}}^2+(1-\lambda) {\norm{y-z}}^2}.
\end{align*}
Notice that
$x-z=(1-\lambda)(x-y)$ and $y-z=\lambda(y-x)$. Thus, substituting in the last inequality we get
\begin{align*}
\om(z)&\le\lambda \om(x) +(1-\lambda )\om(y)-\lambda \innp{\grad \om(z),(1-\lambda)(x-y)} -(1-\lambda) \innp{\grad \om(z),\lambda(y-x)}\\
&  -\frac{\alpha}{2}[\lambda (1-\lambda)^2{\norm{x-y}}^2+(1-\lambda){\lambda}^2{\norm{x-y}}^2]\\
 & =  \lambda \om(x) +(1-\lambda )\om(y) -\lambda (1-\lambda) \innp{\grad \om(z), x-y}+\lambda(1-\lambda)\innp{\grad \om(z),x-y}\\
&  -\frac{\alpha}{2}\lambda (1-\lambda)\bk{ (1-\lambda){\norm{x-y}}^2+{\lambda}{\norm{x-y}}^2}\\
 & =  \lambda \om(x) +(1-\lambda )\om(y) -\frac{\alpha}{2}\lambda (1-\lambda){\norm{x-y}}^2.
\end{align*}
Substituting for $z=\lambda x +(1-\lambda)y$ gives (\ref{i}).

``\eqref{ii}$\RA$\eqref{iii}".
Suppose that (\ref{ii}) holds $\forall x,y \in X$. Then we have
\begin{equation}\label{abv:1}
\om(x)-\om(y)-\innp{\grad \om(y),x-y} \ge  \frac{\alpha}{2}{\norm{x-y}}^2,
\end{equation}
\begin{equation}\label{abv:2}
\om(y)-\om(x) + \innp{\grad \om(x),x-y} \ge  \frac{\alpha}{2}{\norm{x-y}}^2.
\end{equation}
Adding \eqref{abv:1} and \eqref{abv:2} we get (\ref{iii}).

``\eqref{iii}$\RA$\eqref{ii}".
 By the fundamental theorem of calculus we have for $t \in \opint{0,1}$
\begin{equation*}
\om(x)-\om(y)=\int_{0}^{1}\innp{\grad \om(y+t(x-y)),x-y}dt.
\end{equation*}
Subtracting $\innp{\grad \om(y),x-y}$, noting that $\int_{0}^{1}\innp{\grad \om(y),x-y}dt=\innp{\grad \om(y),x-y}$ and using (iii)  we get
\begin{align*}
\om(x)-\om(y)-\innp{\grad \om(y),x-y} &= \int_{0}^{1}\innp{\grad \om(y+t(x-y))-\grad \om(y),x-y} dt\\
&= \int_{0}^{1}\frac{1}{t}\innp{\grad \om(y+t(x-y))-\grad \om(y),t(x-y)} dt\\
& \ge  \int_{0}^{1}\frac{1}{t} \alpha {\norm{t(x-y)}}^2 dt\\
& =  \alpha {\norm{x-y}}^2\int_{0}^{1}\frac{1}{t} t^2  dt\\
&=\frac{\alpha}{2}{\norm{x-y}}^2,
\end{align*}
which completes the proof.
\end{proof}

Following \cite{BBTGBT2010}, we give the definition of the Bregman diameter of an arbitrary set $X$.

\begin{defn}
Let $\om:O\to \RR $ be a convex differentiable function. Let $X$ be a nonempty subset of $O$. The \emph{Bregman diameter of the set $X$} is defined as 
\begin{equation}\label{bregdm}
\Theta= \sup \{D_{\om}(u,v):u,v \in X\}. 
\end{equation}
\end{defn}
In the following lemma we prove that, if $\om$ is differentiable and strongly convex, then the Bregman diameter is finite for every compact subset of $\RR^m$.

\begin{lem}\label{diam}
 Let $\om : O \ra \mathbb{R}$ be a differentiable convex function. Let $X$ be a nonempty compact subset of $O$. Then $D_\omega $ is bounded on $X\times X$. Consequently, the Bregman diameter of the set $X$ is finite.
\end{lem}

\begin{proof}
Since $\om $ is convex and differentiable, therefore $\om$ is continuously differentiable on $O$ 
{\rm\cite[Corollary~25.5.1]{Rocky70}}. Thus, $\om$ and $\grad \om$ are continuous on $X$, and therefore $D_\om$ is continuous on $X\times X$. Now, $X\times X$ is a nonempty compact subset of $\RR^m\times \RR^m$, and therefore $D_\omega $ is bounded on $X\times X$ and the Bregman diameter of the set $X$ is finite.
\end{proof}

\section{Functional Constraints and Assumptions}\label{s:3}

Recall that we are interested in the general convex problem of the form
\begin{equation}\label{P}
(P):\min~\{f(x):g(x) \le 0, x\in X\}.
\end{equation}

In the sequel, we shall consider the following assumptions on $f$, $g$ and $X$.

\begin{description}
\sepp
	\item[A1] $f:O\to \RR$ and $g:O\to \RR$ are continuous convex functions. 
	\item[A2] $X$ is a nonempty compact convex subset of $O$, and $X$ is not a singleton.
	\item[A3] We have access to lower-$\C^{2}$ 
	representations (see Theorem~\ref{save}) of $f$ and $g$ on some open subset $O'$ of $O$ such
	that $X\subseteq O'$ and 
		\begin{equation*}
		f=\ds\max_{t\in T_f} f_t ~~~~~\text{and}~~~~~~ 
		g=\ds\max_{t\in T_g} g_t ~~~\text{on $O'$.}
		\end{equation*}
		
	\item[A4] The set of optimal solutions of problem $(P)$ is nonempty. 
	\end{description}

\begin{rem}\label{assum}
Under Assumption {\bf A1}, the functions $f$ and $g$ are lower-$\C^{2}$ functions on $O$ (by Lemma~\ref{con_lc2}).  Assumption {\bf A3} provides the stronger statement that we have access to lower-$\C^2$ representations of these functions.
\end{rem}

\begin{lem}\label{bdsbgd}
Suppose that Assumptions {\bf A1} and {\bf A2} hold. 
Then 
\begin{equation}\label{gradbd}
\Lf{f}:= \sup\|\pt f(X)\|<+\infty  ~~~\text{and}~~~ 
\Lf{g}:= \sup\|\pt g(X)\|<+\infty. 
\end{equation}
\end{lem}
\begin{proof}
Combine Remark~\ref{assum}, Assumption~\textbf{A2}, and Lemma~\ref{bound:aux}(ii). 
\end{proof}

In the following Theorem, we give an error bound for the approximate subgradient. 
 
\begin{thm}\label{errbdd}
  Suppose that {\bf A1}, {\bf A2}, {\bf A3}, and {\bf A4} hold.
  Let $Y=\tuple\in \matt$ be a poised tuple of sample
  points centred at $y_0\in X$ with respect to $f$. 
  Set
  $ \Delta=\ds\max _{1\le i\le m}\norm{\spy{i}-\spy{0}} $. Suppose that $B
  (\spy[y]{0};\Delta)\subseteq X$. Let $y\in B (\spy[y]{0};\Delta)$. Let
  $(t_1,\ldots,t_r)\in A(y)^r$, and $\lam\in \Cr{r}$, where $r\in\NN$. 
Define $V=V(y):=\sum_{i=1}^{r} \lam_i \grad F_{t_i}(y)$. 
Then there exists $v\in \pt f(y)$ such that the following error bound
holds:
\begin{equation*}
\norm{V-v}\le K_f(1+\sqrt{m}\norm{\hat{L}^{-1}}/2)~ {\Delta},
\end{equation*}
where $K_f$ is as in Lemma~\ref{coffee}, 
and $\hat{L}=\hat{L}(Y)$ is as defined in Theorem~\ref{Lin_int_err}.
\end{thm}

\begin{proof}
 By assumption $V=\sum_{i=1}^{r} \lam_i \grad F_{t_i}(y)$. Lemma~\ref{subgd} implies that 
 $v=v(y):=\sum_{i=1} ^{r}\lam_i \grad f_{t_i}(y)\in \pt f(y)$. Using the
 triangle inequality, the error bound given in Theorem~\ref{Lin_int_err}
 (applied to $O'$ instead of $O$) and Lemma~\ref{coffee}, we have
\begin{align*}
\norm{V-v}&=\norm{\sum_{i=1}^{r}\lam_i( \grad F_{t_i}(y)- \grad f_{t_i}(y))}
\le \sum_{i=1}^{r}\lam_i \norm{\grad F_{t_i}(y)- \grad f_{t_i}(y)}\\
&\le \sum_{i=1}^{r}\lam_i K_f (1+\sqrt{m}\norm{\hat{L}^{-1}}/2) 
= K_f(1+\sqrt{m}\norm{\hat{L}^{-1}}/2) ~\Delta,
\end{align*}
as claimed.
\end{proof}

Our next corollary relates Theorem~\ref{errbdd} to the algorithm presented later.  Let us note that the function $E$ in Corollary~\ref{cor:11} is the same as the one used in the algorithm.  We also note that, although in Corollary~\ref{cor:11} we provide the error bound for the approximate gradient function in a general format, in practice we shall use $x=y_0$.

\begin{cor}\label{cor:11}
Suppose that {\bf A1}, {\bf A2}, {\bf A3} and {\bf A4} hold. 
Let $Y=\tuple$ be a poised tuple of sample points centered at $y_0\in X$ with respect to $f$. 
Set $ \Delta=\ds\max _{1\le i\le m}\norm{\spy{i}-\spy{0}}$ and suppose that
$B (\spy[y]{0};\Delta)\subseteq X$. For every $x\in B(\sby[y]{0};\Delta)$,
let $\bk{t_1,\ldots,t_{r(x)}}\in A_f(x)^{r(x)}$, $\lam\in \Cr{r(x)}$,
$\bk{\bar{t}_1,\ldots,\bar{t}_{\bar{r}(x)}}\in A_g(x)^{\bar{r}(x)}$,
$\bar{\lam}\in \Cr{\bar{r}(x)}$, 
\begin{align*}
v_{f}(x)&=\sum_{i=1} ^{r(x)}\lam_i \grad f_{t_i}(x)\in \pt f(x),
&V_{f} (x)&=\sum_{i=1} ^{r(x)}\lam_i \grad F_{t_i}(x), \\
v_{g}(x)&=\sum_{i=1} ^{\bar{r}(x)}\bar{\lam}_i \grad g_{\bar{t}_i}(x)\in \pt g(x), 
&V_{g} (x)&=\sum_{i=1} ^{\bar{r}(x)}\bar{\lam}_i \grad G_{\bar{t}_i}(x),
\end{align*}
and 
\begin{equation}\label{grad}
e(x):=\begin{cases}
v_f (x) ,~~~& \text{if $g(x) \le \eps$}, \\
v_g (x),~~~& \text{otherwise},
\end{cases}
\end{equation}
and
\begin{equation}\label{aproxgrad}
E(x):= \begin{cases}
V_f (x),~~~& \text{if $g(x)  \le \eps$}\\
V_{g} (x),~~~ & \text{otherwise}.
\end{cases}
\end{equation}
Then:
\begin{enumerate}[\rm(i)]
\sepp
\item 
 The following error bound holds
\begin{equation}\label{EB}
\norm{e(x)-E(x)}\le \kappa~ \Delta , \text{~~~for all~~~ } x\in B(\sby[y]{0};\Delta),
\end{equation} 
where $\kappa=\max\{K_f,K_g\}(1+\sqrt{m}\norm{\hat{L}^{-1}}/2)$, 
$K_f$ is defined as in Lemma~\ref{coffee} and
$K_g$ is obtained by replacing $f$ by $g$ in Lemma~\ref{coffee}, 
and $\hat{L}$ is as defined in Theorem~\ref{Lin_int_err}.
\item 
 The function $E$ induced by \eqref{aproxgrad} satisfies
\begin{equation}\label{EBB}
\norm{E(x)}\le  
\max\stb{L_f,L_g}+\kappa~\Delta, \text{~~~for all~~~ } x\in B(\sby[y]{0};\Delta),
\end{equation}
where $L_f$ and $L_g$ are defined as in Lemma~\ref{bdsbgd}. 
\end{enumerate}
\end{cor}

\begin{proof}
{\rm(i)}: 
Use \eqref{grad} and \eqref{aproxgrad},
and apply Theorem~\ref{errbdd} to $f$ and $g$.
{\rm(ii)}: 
Let $x\in X$. Using the triangle inequality, \eqref{gradbd}, and \eqref{EB} we have
$\norm{E(x)}\le \norm{e(x)}+\norm{e(x)-E(x)}
\le  \max\stb{L_f,L_g}+\kappa~\Delta$.
\end{proof}

\section{Algorithm and Discussion}\label{s:4}
In this section we introduce the Derivative-Free $\eps-$CoMirror algorithm and present a convergence analysis.

\begin{framed}

\begin{description}
\item[\underline{The Derivative-Free $\eps-$CoMirror algorithm ($\text{DFO}_\eps\text{CM}$)}]
\item[Initialization]  Input
\begin{itemize}
\sepp
 \item $x_0 \in X$, 
 \item $M\in \RR_{++}$.
 \end{itemize}
\item[General step] for every $k\in\stb{1,2,\ldots}$
\begin{itemize}
\sepp
\item Select
 \begin{equation}\label{delta:def}
0< \Delta_k \le \ds\frac{1}{\sqrt{k+1}}~.
\end{equation}
\item

Select a poised tuple $\spy[Y]{k}=\tuple$ centred at $y_0$ with respect to $f$ such that the set $\stb{y_0,y_1,\ldots,y_m}\subseteq B(x_k,\Delta_k)$, $\sby[x]{k}=y_0$ and $\norm{\hat{L}_k^{-1}}\le M$, where $\hat{L}_k=\hat{L}(\spy[Y]{k})$ is as defined in Theorem~\ref{Lin_int_err}. 

\item Set
\begin{equation}\label{2.2}
x_{k+1}=\amin_{x\in X} \{\innp{t_kE_k-\grad \omega(x_k),x}+\omega(x)\},
\end{equation}
where 
\begin{equation}\label{aproxgd}
E_k:=\begin{cases}
V_{f}(x_k),~~~& \text{if $g(\sby[x]{k})  \le \eps$;}\\
V_{g}(x_k),~~~ & \text{otherwise},
\end{cases}
\end{equation}
\begin{equation}\label{2.4}
t_k=\frac{\sqrt{\Theta \alpha}  }{\norm{E_k}\sqrt{k}}~,
\end{equation}
and where $\alpha >0$ is the strong convexity parameter of the strongly
convex function $\om\colon O\to\RR$, $\Theta$ is the corresponding 
Bregman diameter of the set $X$, and $V_{f}$ and $V_{g}$ are defined as in Corollary~\ref{cor:11}.
\end{itemize}
\end{description}
\end{framed}

\begin{rem}~
\begin{enumerate}[\rm(i)]

\item In generating the points of the tuple $\spy[Y]{k}\subseteq \matt$ 
we need to check that $\norm{\hat{L}_k^{-1}}\le M$. 
If this inequality fails, then we resample. It is always possible to
generate the tuple $\spy[Y]{k}$ for all $k\in \NN$  provided that $M$ is
set to be sufficiently large {\rm\cite{Wschebor2004}}. For a detailed discussion on how to choose $M$ we refer the reader to {\rm\cite{conn-scheinberg-vicente-2008}}.

\item
The poised tuple $\spy[Y]{k}=\tuple$ must satisfy
$\ds\max_{i\in \{1,\ldots,m\}}\norm{y_i-\sby[x]{k}}\le \Delta_k$ to
guarantee that the error bound in Theorem~\ref{errbdd} still holds true.
This does not create a conflict {\rm(i)} because by the definition of the
matrix $\hat{L}$ in \eqref{mat_L_def}, the value of $\norm{\hat{L}^{-1}}$ remains unchanged under scaling or shifting. 

\item The update of $x_k$ in \eqref{2.2} is well defined, 
since that the function $\innp{t_k E_k-\grad \omega(x_k),\cdot}+\omega$ is strongly convex and differentiable over $X$, and therefore it has a unique minimizer over $X$.

\item The step length $t_k$ is well defined for all $k\in\stb{1,2,\ldots}$
except when $E_k=0$ in which case either we have a local minimum, or we change the search radius $\Delta_k$ to get a better approximation of the gradients. 
Moreover, the Bregman diameter $\Theta$ is finite by Lemma~\ref{diam}. Finally, by Lemma~\ref{equi}~{\rm (ii)}, we have that $D_\om(x,y)\ge \frac{\alpha}{2}\norm{x-y}^2$, 
and therefore, since $X$ is not a singleton, the  Bregman diameter $\Theta$ is strictly positive.  

\item In general, the Bregman diameter $\Theta $ is not easy to calculate. However, if the set $X$ is simple and the function $\om $ is separable, calculating $\Theta$ becomes simpler. 
For example, if $X=[\alpha_1,\beta_1]\times\cdots\times [\alpha_m,\beta_m]$ 
and $\om(x)=\sum_{i=1}^{m}\om_i(x_i)$,
then $\Theta = \sum_{i=1}^{m} D_{\omega_i}(\alpha_i,\beta_i)$. 

\end{enumerate}

\end{rem}

\subsection{Convergence Analysis}

We devote this subsection to study the convergence of the algorithm. Lemma~\ref{key} and its proof are only a minor adaptation of \cite[Lemma~2.2]{BBTGBT2010}. 
For the sake of completeness, we include the adapted proof. 

\begin{lem} \label{key}
Let $\bk{x_k}_{k\in \NN} $ be the sequence generated by \DFOCM. Let $i<j$ be two strictly positive integers. Then for all $k\in \stb{1,2,\ldots}$

\begin{equation}\label{2.6}
\displaystyle \sum_{k=i}^j t_k\innp{E_k,x_k-u} \le \Theta +\frac{1}{2\alpha} \displaystyle \sum_{k=i}^j t_k^2 {\norm{E_k}}^2,
\end{equation}
for every $u \in X$.
\end{lem}

\begin{proof}
By the optimality condition in $\eqref{2.2}$ we have

\begin{equation*}
\innp{t_kE_k-\grad \omega(x_k)+\grad \omega(x_{k+1}),u-x_{k+1}} \ge 0 \text{~for every~} u \in X.
\end{equation*}
Hence,
\begin{equation}\label{2.7}
t_k \innp{E_k,u-x_{k+1}} \ge \innp{\grad \om(x_k)-\grad \om (x_{k+1}),u-x_{k+1} } \text{~for every~} u \in X.
\end{equation}

The three-point property of the Bregman distance {\rm\cite[Lemma 3.1]{Chen-Teboulle}} tells us
\begin{equation}\label{2.8}
D_{\om}(u,x_{k+1})-D_{\om}(u,x_k)+D_{\om}(x_{k+1},x_k)=\innp{\grad \om(x_k)-\grad \om (x_{k+1}),u-x_{k+1} }.
\end{equation}

Combining (\ref{2.7}) and (\ref{2.8}) yields
\begin{equation*}
t_k\innp{E_k,u-x_{k+1}} \ge D_{\om}(u,x_{k+1})-D_{\om}(u,x_k)+D_{\om}(x_{k+1},x_k).
\end{equation*}
That is 
\begin{equation*}
 t_k\innp{E_k,x_{k+1}-u} \le D_{\om}(u,x_k)-D_{\om}(x_{k+1},x_k)-D_{\om}(u,x_{k+1}).
\end{equation*}
Adding $t_k\innp{E_k,x_k-x_{k+1}}$ to both sides of the above inequality and using Lemma~\ref{equi} {\rm(ii)} and the Cauchy-Schwarz inequality we get
\begin{align*}
t_k\innp{E_k,x_k-u}& \le  D_{\om}(u,x_k)-D_{\om}(u,x_{k+1}) -D_{\om}(x_{k+1},x_k)+t_k\innp{E_k,x_k-x_{k+1}}\\
&\le  D_{\om}(u,x_k)-D_{\om}(u,x_{k+1}) -\frac{\alpha}{2}{\norm{x_k-x_{k+1}}}^2+t_k{\norm{E_k}}\norm{x_k-x_{k+1}}.
\end{align*}
Notice that, $t_k{\norm{E_k}}\norm{x_k-x_{k+1}}-\frac{\alpha}{2}{\norm{x_k-x_{k+1}}}^2$ is a quadratic function of $\norm{x_k-x_{k+1}}$ that has a maximum value of $\frac{1}{2 \alpha}t_k^2{\norm{E_k}}^2$,  
i.e., $t_k{\norm{E_k}}\norm{x_k-x_{k+1}}-\frac{\alpha}{2}{\norm{x_k-x_{k+1}}}^2 \le \frac{1}{2 \alpha}t_k^2{\norm{E_k}}^2$. This yields
\begin{equation*}
t_k\innp{E_k,x_k-u} \le D_{\om}(u,x_k)-D_{\om}(u,x_{k+1})+ \frac{1}{2 \alpha}t_k^2{\norm{E_k}}^2.
\end{equation*}
Summing the last inequality over $k\in\stb{i,i+1,\ldots,j} $ we obtain

\begin{equation*}
\displaystyle \sum_{k=i}^{j}t_k\innp{E_k,x_k-u}\le D_{\om}(u,x_i)-D_{\om}(u,x_{j+1})+ \displaystyle \sum_{k=i}^{j} \frac{1}{2 \alpha}t_k^2{\norm{E_k}}^2.
\end{equation*}

Using the definition of $\Theta$ we note that  $D_{\om}(u,x_i)-D_{\om}(u,x_{j+1}) \le \Theta$, from which we get (\ref{2.6}).
\end{proof}

The following theorem presents the efficiency estimate for the Derivative-Free $\eps-$CoMirror method. In proving Theorem~\ref{t:conv} we are motivated by the techniques used in the proof of \cite[Theorem~ 2.1]{BBTGBT2010}.
Given $n\in \NN$, we denote the set of indices of the $\eps-$feasible solutions among the first $n$ iterations by
\begin{equation*}
I_n^\eps=\stb{k\in\stb{1,2,...,n}:g(x_k)\le \eps}. \end{equation*}

\begin{thm}\label{t:conv}
Suppose that Assumptions {\bf A1}, {\bf A2}, {\bf A3} and {\bf A4} hold. 
Let $\eps>0$ and let  $\bk{x_k}_{k\in \NN} $ be the sequence generated by
\DFOCM. Denote by $\fopt$ the optimal function value of \eqref{P}. Then for every $n\in \stb{4,5,\ldots}$
	\[\min  \left\{ \displaystyle \min_{k \in I^\eps _n}\bk{f(x_k)-\fopt},\eps\right\} \le  \frac{C}{\sqrt{n}},\]
where 
	\[C=\displaystyle  2\sqrt{\frac{\Theta}{\alpha}} \max \stb{\kappa_1,\kappa_2}\ds \frac{1+ \ln (2)}{2-\sqrt{2}} + \kappa_2~\Omega,\]
	\[\kappa_1= \max\stb{L_f,L_g},\]
	\[\kappa_2=K(1+\sqrt{m}M/2),\]
	\[\Omega=\max_{x,y\in X} \norm{x-y} ,\]
 $L_f$ and $L_g$ are as defined in \eqref{gradbd}, $K$ is as defined in Corollary~\ref{cor:11}, and $M>0$ satisfies that $\norm{\hat{L}_k^{-1}}\le M$ for all $k\in \stb{1,2,\ldots}$.
\end{thm}

\begin{proof}
Using assumption {\bf A4}, suppose that $\xopt$ is an optimal solution
of (\ref{P}). Fix $n\in \stb{1,2,\dots}$, and $k\in \stb{1,2,\ldots,n}$. 
We begin by considering the following two cases:
\begin{description}
\item[Case I:]
 $ k \in I^ {\eps} _n  $. 
Then $g(x_k)\leq\eps$, 
and, by $\eqref{aproxgd}$, $\eqref{grad}$, and \eqref{aproxgrad} we have $e_k:=e(x_k)=v_f(x_k)\in \pt f(\sby[x]{k})$ and $E_k:=E(x_k)= V_f({x_k}) $, and hence 
\begin{equation*}
 f(x_k)\le f(\xopt)+\innp{e_k,x_k-\xopt} .
\end{equation*}
Therefore, using Cauchy-Schwarz inequality and the error bound in equation $\eqref{EB}$
\begin{align*}
f(x_k)&\le f(\xopt)+\innp{E_k,x_k-\xopt}+\innp{e_k-E_k,x_k-\xopt}\\
&\le f(\xopt)+\innp{E_k,x_k-\xopt}+\norm{e_k-E_k}\norm{x_k-\xopt}\\
&\le f(\xopt)+\innp{E_k,x_k-\xopt}+\kappa_2~ \Delta_k~\Omega.
\end{align*}
Hence
\begin{equation} \label{2.13}
f(x_k)-f(\xopt)\le \innp{E_k,x_k-\xopt}+\kappa_2~ \Delta_k~\Omega. 
\end{equation}
\item[Case II:]
$ k \not\in I^ {\eps} _n$. Then $g(x_k) > \eps$. 
Using $\eqref{aproxgd}$, $\eqref{grad}$, and \eqref{aproxgrad} we have $  e_k=v_g(x_k) \in \pt g (\sby[x]{k})  $ and $ E_k=V_g(x_k)$, and hence
\begin{equation*}
  g(x_k) \le g(\xopt)+\innp{e_k,x_k-\xopt}.
\end{equation*}
Since $g(\xopt) \le 0$ we have
\begin{align*}
\eps &<  g(x_k) \\
&\le  g(\xopt)+\innp{e_k,x_k-\xopt}\\
& \le\innp{e_k,x_k-\xopt}= \innp{E_k,x_k-\xopt} +\innp{e_k-E_k,x_k-\xopt} .
\end{align*}
Hence, using Cauchy-Schwarz inequality, the assumption that $\norm{\hat{L}_k^{-1}}\le M$ for all $k\in \stb{1,2,\ldots}$, and the error bound in equation $\eqref{EB}$ we have
\begin{align}
\eps& \le  \innp{E_k,x_k-\xopt} +\norm{e_k-E_k}\norm{x_k-\xopt}\nonumber\\
&\le \innp{E_k,x_k-\xopt} +\kappa_2~ \Delta_k~\Omega\label{2.14}.
\end{align}
\end{description}
By combining Case I and Case II, we have 
\begin{equation}\label{2.15}
\innp{E_k,x_k-\xopt}+\kappa_2~ \Delta_k~\Omega \ge
\begin{cases}
 f(x_k) -f(\xopt), & \text{~~if~~} k \in I_n^\eps, \\
  \eps, & \text{~~if~~} k \not\in I_n^\eps.
\end{cases} 
\end{equation}

Using (\ref{2.15}) we have for all $1\le l\le n$, with $\Delta_l\le 1/ \sqrt{l+1}$
\begin{equation*}
\min \{ {\min_{k\in I_n^{\eps}} \bk{f(x_k)-f(\xopt)},\eps} \} \le \innp{E_l,x_l-\xopt}+\kappa_2~ \Delta_l~\Omega.
\end{equation*}
Let $n_0\in \stb{1,2,\ldots,n}$, then using \eqref{delta:def}
\begin{align}
\min \stb{ \min_{k\in I_n^{\eps}}\bk{ f(x_k)-f(\xopt)},\eps } 
&\le \dsm_{n_0 \le l \le n} \bk{ \innp{E_l,x_l-\xopt} +\kappa_2~ \Delta_l~\Omega} \nonumber \\  
&\le \dsm_{n_0 \le l \le n} \bk{ \innp{E_l,x_l-\xopt} +\kappa_2~ \Omega \ds\max_{n_0 \le l \le n}\Delta_l }  \nonumber\\ 
& \le \dsm_{n_0 \le l \le n} \bk{ \innp{E_l,x_l-\xopt} }+\frac{\kappa_2~ \Omega}{\sqrt{n_0+1}} \label{eq:1}.
\end{align}

Substituting $u=\xopt,~i=n_0,~j=n$ in   Lemma~\ref{key} we see that 
\begin{equation}\label{2.11}
\dss _{k=n_0}^n t_k\innp{E_k,x_k-\xopt} \le \Theta +\dsf{1}{2\alpha} \dss_{k=n_0}^n t_k^2{\norm{E_k}}^2.
\end{equation}
On the other hand, since $X$ is not a singleton, we have $t_k>0$ for every $k\in\stb{1,2,\dots,n}$, and thus
\begin{equation}\label{2.11:r}
\dss_{k=n_0}^n  t_k\innp{E_k,x_k-\xopt} \ge \bk{\dsm_{n_0\le k \le n} \innp{E_k,x_k-\xopt}} \dss_{k=n_0} ^n t_k.
\end{equation}

Combining \eqref{2.11} and \eqref{2.11:r} yields
\begin{equation}\label{2.12}
 \dsm_{n_0\le k \le n} \innp{E_k,x_k-\xopt} \le \frac{ \Theta +\dsf{1}{2\alpha} \dss_{k=n_0}^n t_k^2{\norm{E_k}}^2}{\dss_{k=n_0} ^n t_k} .
\end{equation}

Using \eqref{2.4}, we have
\begin{equation}\label{ine:1}
 \dss_{k=n_0}^n t_k^2{\norm{E_k}}^2=\Theta\alpha\dss_{k=n_0}^n\frac{1}{k}, 
 \end{equation}

 and  
 \begin{equation}\label{w:geded}
 \dss_{k=n_0} ^n t_k=\sqrt{\Theta \alpha}\dss_{k=n_0} ^n \frac{1}{\norm{E_k} \sqrt{k}}.
 \end{equation}
We recall that $\norm{\hat{L}_k^{-1}}\le M$ for all $k\in \stb{1,2,\ldots}$, $\kappa_1= \max\stb{L_f,L_g}$ and  $\kappa_2=K(1+\sqrt{m}M/2)$. Now, for every $k\in\stb{1,2,\dots}$ using Corollary~\ref{cor:11} and \eqref{delta:def} we have  
\begin{align}
\norm{E_k} \sqrt{k}&\le  (\kappa_1+\kappa_2~\Delta_k )\sqrt{k} 
\le \kappa_1 \sqrt{k}+\kappa_2\frac{\sqrt{k}}{\sqrt{k+1}} \nonumber\\
&\le \kappa_1 \sqrt{k}+\kappa_2
\le  \max \stb{\kappa_1,\kappa_2 }(\sqrt{k}+1) \nonumber\\
&\le  2\max \stb{\kappa_1,\kappa_2 }\sqrt{k}. \label{eq:geded}
\end{align}
Using \eqref{w:geded} and \eqref{eq:geded} we get 
\begin{equation}\label{ine:2}
\dss_{k=n_0} ^n t_k\ge\frac{\sqrt{\Theta\alpha}}{2\max \stb{\kappa_1,\kappa_2 }}\dss_{k=n_0} ^n\frac{1}{\sqrt{k} }, 
 \end{equation} 
Using equations $\eqref{ine:1}$ and $\eqref{ine:2}$, inequality \eqref{2.12} becomes
\begin{equation}\label{eq:5}
 \dsm_{n_0\le l \le n} \innp{E_k,x_k-\xopt}  \le\frac{ 2\Theta \max\stb{\kappa_1,\kappa_2}\bk{1+\frac{1}{2} \dss_{k=n_0}^n\frac{1}{k}}}{\sqrt{\Theta \alpha}\dss_{k=n_0} ^n\frac{1}{\sqrt{k} }} .
\end{equation}
Now, set $n_0=\left \lfloor n/2 \right \rfloor$. On the one hand, using \eqref{eq:5} and Lemma~\ref{inequalities} we get
\begin{equation}\label{eq:3}
\dsm_{n_0 \le l \le n} \innp{E_l,x_l-\xopt} \le \frac{C_1}{\sqrt{n}},
\end{equation}
where $C_1= 2\sqrt{\frac{\Theta}{\alpha}} \max \stb{\kappa_1,\kappa_2}\ds \frac{1+ \ln (2)}{2-\sqrt{2}}$. On the other hand, using the fact that $\left \lfloor n/2 \right \rfloor +1 > n/2$ we have
\begin{equation}\label{eq:4}
\frac{\kappa_2~\Omega}{\sqrt{n_0+1}}=\frac{\kappa_2~\Omega}{\sqrt{\left \lfloor n/2 \right \rfloor +1}}\le \frac{C_2}{\sqrt{n}} ,
\end{equation}
 where $C_2=\sqrt{2}~ \kappa_2~\Omega$.
Using \eqref{eq:3} and $\eqref{eq:4}$ we deduce that  
\begin{equation}
\min \{ {\min_{k\in I_n^{\eps}} f(x_k)-f(\xopt),\eps} \} \le \frac{C_1+C_2}{\sqrt{n}}= \frac{C}{\sqrt{n}},
\end{equation}
which completes the proof.
\end{proof}

\section{Numerical Results}\label{s:5}

In this section we provide some numerical results of the \DFOCM~ algorithm.  The \DFOCM~ algorithm was implemented in MATLAB.  To begin we examine three academic test problems from \cite{correia-matias-mestre-serodio-2009, correia-matias-mestre-serodio-2010}.  We then apply the \DFOCM~ algorithm to a simulation test problem from \cite{hare-2010}.

\subsection{Academic Test Problems}
We first consider three academic test problems from \cite{correia-matias-mestre-serodio-2009, correia-matias-mestre-serodio-2010}.  In working with these problems, we rewrite the constraint functions as a single constraint via a max function.  For example, in $\probi$ the constraint functions are rewritten as $g(x_1,x_2)=\ds\max_{1\le i\le 3} g_i(x)$, where $g_1(x_1,x_2)=-x_1$, $g_2(x_1,x_2)=x_1-1$ and $g_3(x_1,x_2)=x_2$.

\begin{enumerate}[\rm(i)]

\item $\probi$ 
\begin{align*}
(x\in \RR^2) \text{~~Minimize~~}&-x_1-2x_2\\
\text{subject to~~}&0\le x_1\le 1\\
&x_2\le 0.
\end{align*}

\item $\probii$ 
\begin{align*}
(x\in \RR^2) \text{~~Minimize~~}&6x_1^2+x_2^2-60x_1-8x_2+166\\
\text{subject to~~}&0\le x_1\le 10,\\
&0\le x_2\le 10,\\
&x_1+x_2-x_1x_2\ge 0,\\
&x_1+x_2-3\ge 0.
\end{align*}

\item $\probiii$ 
\begin{align*}
(x\in \RR^2) \text{~~Minimize~~}&7x_1^2+3x_2^2-84x_1-34x_2+300\\
\text{subject to~~}&0\le x_1\le 10,\\
&0\le x_2\le 10,\\
&x_1x_2-1\ge 0,\\
&9-x_1^2-x_2^2\ge 0.
\end{align*}
\end{enumerate}

\begin{rem}
In {\rm\cite{correia-matias-mestre-serodio-2009}} and {\rm\cite{correia-matias-mestre-serodio-2010}}, the authors mention that their algorithms could not find an optimal solution to $\probiii$.  This is due to them incorrectly stating that the optimal value is $-97.30952$.  The correct optimal value is $\fopt\approx 84.6710$, which we demonstrate below.

Define $f$, $g_1$, and $g_2$ as follows,
\[\begin{array}{rcl}
	f(x_1,x_2)&=&7x_1^2+3x_2^2-84x_1-34x_2+300,\\
	g_1(x_1,x_2)&=&1-x_1x_2\le 0, ~~\mbox{and}\\
	g_2(x_1,x_2)&=&x_1^2+x_2^2-9\le 0.
\end{array}\]
Notice that $f(x_1,x_2)=7(x_1-6)^2+3(x_2-\frac{17}{3})^2-\frac{145}{3}$, so $f$ is strictly convex. \\
The constraint set 
$\stb{(x_1,x_2)\in \RR^2\colon 0\le x_1\le 10,~0\le x_2\le 10, ~ g_1(x_1, x_2) \leq 0 \text{~and~}  g_2(x_1, x_2) \leq 0}$ is also convex. Let $a$ be the positive real root of $p(x)=16 x^4-336x^3+1909x^2+3024x-15876$. Then at $x_1=a$, and $x_2=\frac{8}{357}a^3-\frac{4}{17}a^2+\frac{145}{714} a+\frac{36}{17}$, with $\lam=-1-\frac{1909}{378}a+\frac{8}{9}a^2-\frac{8}{189}a^3$ we have $1-x_1x_2<0$, $x_1^2+x_2^2=9$ and $\grad f(x_1,x_2)=\lam \grad g_2(x_1,x_2)$; that is first order optimality holds. As the objective function and constraint set are convex, this implies optimality.  The corresponding optimal value is $\fopt\approx 84.6710$. Approximate values of $(x_1,x_2)= (2.6390,1.4267)$ and $\lam\approx -8.9150$.
\end{rem}

We test \DFOCM~ on each of these three test problems using two options for
creating the Bregman distance.  In the results of these test problems we
shall use $\omega_1=\frac{1}{2}\norm{\cdot}^2$,  and $\omega_2(x)$ to
denote the (negative) entropy 
$\sum_{i=1}^{m}(x_i) \ln (x_i)$. In Table~\ref{table:T1} we compare our results of the first three test problems to the results obtained by the Pattern Search method and Simplex Search method introduced in \cite{correia-matias-mestre-serodio-2009}. Note that, although in test problems 2 and 3 the constraint functions are non convex, the generated constraint set is convex.  This is not covered by Theorem~\ref{t:conv}, however; the \DFOCM~  still gives a good fit. 

\begin{table}[ht]
\caption{Comparing results for Test Problems 1, 2, and 3.}\label{tab:academictests}
\centering 
\begin{tabular}{c l c c c c} 
\hline\hline 

Test Problem&    Results & \multicolumn {2}{c}{DFO CoMirror }& Pattern Search  &Simplex Search \\ [0.5ex] 
 & &\small{$\omega_1(x)$  }& \small{$\omega_2(x)$}& Algorithm\cite{correia-matias-mestre-serodio-2009} & Algorithm\cite{correia-matias-mestre-serodio-2009}\\
\hline 
1 &Function value &   -0.9542 &  -0.9645 & -1& -1 \\
& $f$ evaluations   &  78  &  99   & 195&158 \\
&$g$ evaluations   & 162&  141 & 157 & 129 \\
\hline 
2&Function value &7.5587 &7.5580 & 7.625& 7.625 \\
&$f$ evaluations  & 78 & 81 & 138 &146  \\
&$g$ evaluations   & 122& 111& 138 & 118 \\
\hline
3&Function value & 84.7096 &   84.7108& 85.6610& 85.6200 \\
& $f$ evaluations   & 78 & 75 & 154 &198  \\
&$g$ evaluations   & 122& 125& 154 & 153 \\
\hline 
\end{tabular}
\label{table:T1} 
\end{table}

Examining Table~\ref{table:T1}, we note that \DFOCM~ outperformed both the Pattern Search and Simplex Search algorithms on Test Problems 2 and 3.  On Test Problem 1, \DFOCM~ did not preform as well, but still required noticeably less function evaluations that the Pattern Search and Simplex Search methods. 

\subsection{Simulation Test Problem}
In this section we test the algorithm on 12-dimensional simulated maximization problem given in \cite{hare-2010}. We used the same staring points given in \cite{hare-2010}: $x_0=(1,0,\ldots,0)$ and $\bar{x}_0=(2,0.5,\ldots,0.5)$ are vectors in $\RR^{12}$. The results are reported in Table~\ref{table:T8}. We compare our results to the results obtained from the Direct Pattern Search Method (DPS) and the Direct Random Search Method with Simulated Annealing (DRS+SA) in \cite{hare-2010}. As the constraint set for this problem is a system of linear inequalities, the methods used in \cite{hare-2010} used exact gradients when dealing with constraints.  Objective function evaluations are provided via deterministic simulation.

The results in \cite{hare-2010} report that, using $3000$ function calls, the DPS gives an optimal value of $0.8327$ with $x_0$ as starting point and an optimal value of $0.1747$ with $\bar{x}_0$ as starting point.  Whereas, using $3000$ function calls, the heuristic DRS+SA gives an optimal value of $0.9628$ with $x_0$ as starting point and an optimal value of $0.9671$ with ${\bar{x}}_0$ as starting point. 

\begin{table}[ht]
\caption{Results of DFO CoMirror algorithm } 

\centering 
\begin{tabular}{l | c c c | c  c c } 
\hline\hline 
 & \multicolumn{3}{c}{Starting point $x_0$}& \multicolumn{3}{c}{Starting point $\bar{x}_0$} \\
f calls & $\eps=0.01$ & $\eps=0.005$ & $\eps=0.001$ & $\eps=0.01$ & $\eps=0.005$ & $\eps=0.001$ \\ [0.5ex] 
\hline 
100 		& 0.7329&0			&   0		& 0.8875     &0						&   0.8968		 \\
500 &	 0.9400&	0.9387&   0.9342  &	 0.9220			&	0.9210		&   0.8332  \\
1000 
      &0.9452&0.9514 &   0.9447 &0.9277			&0.9256 			&   0.9334\\
3000
 & 0.9547&    0.9551&  0.9546  &  0.9500			&    0.9467		&  0.9538 \\
 \hline 
\end{tabular}
\label{table:T8} 
\end{table}

In Table~\ref{table:T8} we see that with 500 function calls, \DFOCM~ is able to achieve a significantly better fit than the DPS.  While the fit for \DFOCM~ never quite achieves the quality of the DRS+SA method, it comes quite close after 3000 function calls.  This difference could be explained by the fact that the DRS+SA method employs heuristics to break free of local minimizers.

\section{Conclusion}\label{s:conc}
In this paper we developed the convergence analysis required to generate a
derivative-free comirror algorithm, \DFOCM.  Furthermore, we provided some
numerical results from the implementation of the algorithm in MATLAB.  One
natural line of future research is to adapt the algorithm to deal with the
problem 
\begin{equation}\label{P0}
(P1): \min\{f(x):g(x) \le 0\},
\end{equation}
i.e., $X=\RR^m$, and to prove convergence. 
Another line of future research is examining the convergence in the case
where $g$ is not necessarily convex, but the constraint set remains convex. 
Results from test problems 2 and 3 suggest that this is possible. 

\appendix
\section{Appendix}\label{s:app}

 \begin{lem} \label{inequalities}
 For any integer $n\in \stb{4,5,\dots}$ the following inequalities hold true
 \begin{equation}\label{inequality1}\dss_{k=\left \lfloor n/2 \right \rfloor}^n \frac{1}{k} \le 2 \ln(2),\end{equation}
 \begin{equation}\label{inequality2} \dss_{k=\left \lfloor n/2 \right \rfloor}^n  \frac{1}{\sqrt{k}} \ge (2-\sqrt{2})\sqrt{n}.\end{equation}
 \end{lem}
\begin{proof} To see inequality \eqref{inequality1}, notice
\begin{align}\nonumber
\dss_{k=\left \lfloor n/2 \right \rfloor}^n \frac{1}{k} &\le \dss_{k=\left \lfloor n/2 \right \rfloor-1}^{n-1}   \ds\int_{k}^{k+1} \frac{1}{x}dx\\ \nonumber
&= \ds \int_{\left \lfloor n/2 \right \rfloor-1}^{n}\dsf{1}{x}dx\\ 
&=\ln \bk{\dsf{n}{\left \lfloor n/2 \right \rfloor-1}}. \label{a:a}
\end{align}
 We now consider two cases ($n$ is even and $n$ is odd).
Case I: suppose $n=2m$ with $m\in \stb{1,2,\dots}$. Then 
\begin{equation}\label{A.3}
	{\dsf{n}{\left \lfloor n/2 \right \rfloor-1}}\le 4 \iff \ds\frac{2m}{m-1}\le 4 \iff n=2m\ge 4.
\end{equation}
Case II:  suppose $n=2m+1$ with $m\in \stb{1,2,\dots}$. Then 
\begin{equation}\label{A.4}
{\dsf{n}{\left \lfloor n/2 \right \rfloor-1}}\le 4\iff\ds\frac{2m+1}{m-1}\le 4\iff n=2m+1\ge 7.
\end{equation}
 Moreover, for $n=5$ direct computation shows that $	{\dsf{n}{\left \lfloor n/2 \right \rfloor-1}}\le 4$, which together with \eqref{a:a}, \eqref{A.3} and \eqref{A.4} proves the first inequality for all $n\in \stb{2,3,\dots}$.

Finally, 
\begin{align*}
\dss_{k=\left \lfloor \frac{n}{2}\right\rfloor}^{n} \frac{1}{\sqrt{k}}  
&=\dss_{k=\left \lfloor  \frac{n}{2}\right\rfloor}^{n} \frac{1}{\sqrt{k}}(k+1-k)
\ge \dss_{k=\left \lfloor  \frac{n}{2}\right\rfloor}^{n} \ds\int_{k}^{k+1}\frac{1}{\sqrt{x}}dx
=\ds\int_{\left \lfloor  \frac{n}{2}\right\rfloor}^{n+1}\frac{1}{\sqrt{x}}dx\\
&\ge\ds\int_{\frac{n}{2}}^{n}\frac{1}{\sqrt{x}} dx= (2-\sqrt{2})\sqrt{n}.
\end{align*}
which proves inequality \eqref{inequality2}
\end{proof}


\section*{Acknowledgments}
HHB was partially supported by the Natural Sciences and
Engineering Research Council of Canada and by the Canada Research Chair
Program.
WLH was partially
supported by the Natural Sciences and Engineering Research Council
of Canada and UBC Internal Research Funding.
WMM was partially supported
by the Natural Sciences and Engineering Research Council
of Canada and UBC Internal Research Funding.

\bibliographystyle{plain}

\end{document}